\documentclass[11pt]{amsart}
\usepackage{color} 
\usepackage{url}
\usepackage[cmtip, all]{xy}

\usepackage[export]{adjustbox}
\usepackage{graphicx}

\usepackage[normalem]{ulem}

\theoremstyle{plain}
\newtheorem{thm}{Theorem}[section]

\newtheorem{lemma}[thm]{Lemma}
\newtheorem{corollary}[thm]{Corollary}
\newtheorem{proposition}[thm]{Proposition}
\theoremstyle{definition}
\newtheorem{remark}[thm]{Remark}

\newtheorem{notations}[thm]{Notations}

\newtheorem{assumption}[thm]{Assumption}

\numberwithin{equation}{section}

\newcommand{\ml}[2]{\begin{multline}\label{#1}#2 \end{multline}}
\newcommand{\ga}[2]{\begin{gather}\label{#1}#2 \end{gather}}

\newcommand{\sF}{{\mathcal F}}

\newcommand{\sO}{{\mathcal O}}

\newcommand{\sU}{{\mathcal U}}
\newcommand{\sV}{{\mathcal V}}

\newcommand{\sX}{{\mathcal X}}

\newcommand{\C}{{\mathbb C}}

\newcommand{\F}{{\mathbb F}}
\newcommand{\G}{{\mathbb G}}

\newcommand{\N}{{\mathbb N}}
\newcommand{\PP}{{\mathbb P}}
\newcommand{\Q}{{\mathbb Q}}

\newcommand{\T}{{\mathbb T}}

\newcommand{\Z}{{\mathbb Z}}

\title [Cohomological dimension in pro-$p$-towers]{Cohomological dimension in pro-$p$ towers}
\author{H\'el\`ene Esnault } 
\address{Freie Universit\"at Berlin, Arnimallee 3, 14195, Berlin,  Germany}
\email{esnault@math.fu-berlin.de}
\thanks{The   author is supported by  the Einstein program}

\date{ June 13, 2018}
\begin{document}
\begin{abstract}
We give a proof without use of perfectoid geometry of the following vanishing theorem of Scholze:  for $X\subset \mathbb{P}^n$ a projective scheme of dimension $d$ over an algebraically closed characteristic $0$ field, and $X_r$ the inverse image of $X$ via the map which assigns $(x_0^{p^r}: \dots: x_n^{p^r})$  to the homogeneous coordinates $(x_0:\ldots :x_n) $, the induced map $H^i(X, \F_p)\to H^i(X_r, \F_p)$ on \'etale cohomology  dies for $i>d$ and $r$ large.  Our proof holds in characteristic $\ell \neq p$ as well.

\end{abstract}
\maketitle

\section{Introduction}\label{intro}

If $X$ is a  proper scheme of finite type  of dimension $d$ defined over an   algebraically closed field $k$  of characteristic  $p>0$, Artin-Schreier theory implies that the cohomological dimension of 
\'etale cohomology of $X$  with $\F_p$-coefficients is  at most $d$, i.e. $H^i(X, \F_p)=0$ for $i>d$.  If $k$ has characteristic not equal to $p$, the cohomological dimension of \'etale cohomology of $X$  with $\F_p$-coefficients  is $2d$ if $X$ is proper, and, by Artin's vanishing theorem, 
 at most $d$ if $X$ is affine.  However, when $X$ is projective,  Peter Scholze showed that there is a specific tower of  $p$-power  degree covers of $X$ which makes its cohomological dimension at most  $d$ in the limit. 

Let $X\subset \PP^n$ be a projective scheme of dimension $d$. We choose coordinates $(x_0:\ldots: x_n)$ on $\PP^n$. With this choice of coordinates, we define the covers  $$\Phi^n_r: \PP^{n} \rightarrow \PP^{n}, \ (x_0:\ldots:x_n) \mapsto (x_0^{p^r}:\ldots:x_n^{p^r}).$$ We define $X_r$ as the inverse image of $X$ by $\Phi^n_r$.

\begin{thm}[Scholze, \cite{Sch14},  Theorem~17.3] \label{thm:scholze}
If $k$ is an algebraically closed field of characteristic $0$, for $ i>d$, one has 
$$\varinjlim_r H^{i}(X_r, \F_p) = 0 .$$

\end{thm}
Scholze obtains the theorem as a corollary of his theory of perfectoid spaces. 
He does not detail the  proof in {\it loc. cit.}, but his argument is documented in \cite{Sch15}.  
By smooth base change, we may assume that $k=\bar \Q_p$. By the comparison theorem \cite[Thm.~IV.2.1]{Sch15}, 
$\varinjlim_r H^r(X_r, \F_p)\otimes \sO_C/p$ is `almost' equal to $H^i(\sX, \sO_{\sX}^+/p)$ where $\sX$ is a perfectoid space he constructs,  associated to $\varprojlim_r X_r$, and $C=\hat {\bar{ \Q}}_p$. By \cite[Thm.~4.5]{Sche92}, the spectral space $\sX$ has cohomological dimension at most the Krull dimension of $X$.

 The aim of this short note is to give an elementary proof, as was asked for over $\C$ in \cite[Section~17]{Sch14}. It turns out that the proof holds in characteristic not equal to $p$ as well.  One obtains the following theorem.
 \begin{thm} \label{thm:main}
If $k$ is an algebraically closed field of characteristic not equal to $p$, for $ i>d$, one has 
$$\varinjlim_r H^{i}(X_r, \F_p) = 0 .$$

\end{thm}
 
 The ingredients are  constructibility and   base change properties  for relative \'etale cohomology with compact supports, functoriality, and some easy facts of representation theory of a cyclic group of $p$-power order. \\[.5cm]

{\it Acknowledgements:}  We thank Deepam Patel who asked us about Theorem~\ref{thm:scholze}. He arose our interest in the problem and 
we had  interesting discussions. We thank Peter Scholze  for explaining to us his perfectoid proof reproduced in the introduction, and for  his  friendly comments on our proof.  We thank the two referees for their work, their friendly remarks which helped us sharpening our note, and for their warm words.

\section{General reduction} \label{sec:red}
As all cohomologies considered are \'etale cohomology with coefficients in $\F_p$, we drop $\F_p$ from the notation whenever it does not create confusion. \\[.4cm]

As \'etale cohomology  only depends  on the underlying reduced structure, we may assume that $X$ is reduced.

We first observe that  Theorem~\ref{thm:main} is true for all $X$ if and only if it is true for all $X$ which are irreducible.  We argue by induction on the dimension $d$ of $X$ and its number $s$ of components. 
If $X$  has dimension $0$,  its cohomological dimension is $0$ and there is nothing to prove. If $X$ has only one component, there is nothing to prove by assumption.  
If $X$ has   $s \ge 2$ components, then it is the union of  $X_1$ and $ X_2$ where  $X_2$ is irreducible,  is not contained in  $X_1$, and $X_1$ has $(s-1)$ components.  Then $X_1\cap X_2$ had dimension $\le (d-1)$.  The Mayer-Vietoris exact sequence 
\ga{}{ \ldots \to H^{i-1}(X_1\cap X_2) \to H^i(X_1\cup X_2)\to H^i(X_1)\oplus H^i(X_2)\to H^i(X_1\cap X_2)\to  \ldots \notag}
shows that  $H^i(X_1)= H^i(X_2)= H^{i-1}(X_1\cap X_2)=0$ for $i> d$, implies $H^i(X_1\cup X_2)=0$ for $i>d$.   But $H^i(X_2)=0$ for $i>d$ by assumption, $H^i(X_1)=0$ for $i>d$ by induction on the number of components, $H^{i-1}(X_1\cap X_2)=0$ for $i>d$ by induction on $d$.

Let $H_a \subset \PP^n$ denote the hyperplane defined by $x_a = 0$, and $Y_a = H_a \cap X$.
If there is one $a$ such that  the dimension of $ Y_a$ is $d$, then 
$X\cap Y_a = X$
as $X$ is irreducible, and one replaces in the statement and the proof $\PP^n$ by $H_a=\PP^{n-1}$. 
So we may assume that 
$Y = \bigcup Y_a$ is a divisor on $X$. 
\begin{notations} \label{nota} Assuming   $Y = \bigcup Y_a$ is a divisor on $X$, 
 let $U \subset X \setminus Y$ be open and dense, $Z=X\setminus U\supset Y$ be the boundary closed subscheme. We let $U_r$ resp. $Y_r$ resp. $Z_r$ denote the pull-back of $U$, resp.  $Y$, resp. $Z$  along $\Phi^n_r$. 
 \end{notations} Then $U_r = X_r \setminus Z_r$  is open dense, and $Z_r$ is closed in $X_r$ of smaller dimension.  The morphism $\Phi_r^n$ 
restricted to $U_r$  is proper and \'etale, thus   the direct system $ \varinjlim_r H^{i}_{c}(U_r)$ of \'etale cohomology with compact supports and coefficients $\F_p$ is defined. 

From the excision sequence
\ga{}{\ldots  \to H^{i-1}(Z_r)\to H^i_c(U_r)\to H^i(X_r)\to H^i(Z_r)\to \ldots \notag}
and induction on the dimension, 
one deduces that  the theorem is true if and only if   $\varinjlim_r H_c^i(U_r)=0$ for $i>d$. 

On the other hand, for $d=n$ then $X=\PP^n$,   $H^{2i}(\PP^n)=\F_p\cdot [L]$ where $L$ is linear of codimension $i$. Thus 
$
\Phi_1^{n*}[L]=\F_p\cdot p^{i}[L]$ which is equal to $0$ as soon as $i>0$. \\[.1cm]

So throughout the rest of the note, we make the general 
\begin{assumption} \label{ass}
 $X$ is an irreducible, reduced projective variety  of dimension $d$ with $0<d<n$ over an algebraically closed field $k$ of characteristic not equal to $p$.
 \end{assumption}
 With Notations~\ref{nota},  we want to draw the 
conclusion
$$\varinjlim_r H^i_c(U_r)=0 \ {\rm  for \ all } \ i > d. $$

\section{Local systems}
\subsection{Geometry preparation} \label{ss:geom}
We define the torus  $ \T= \PP^n\setminus \cup_{i=0}^n H_i$.  It has coordinates 
$(\frac{x_1}{x_0},\ldots, \frac{x_n}{x_0})$. We denote by $\phi_r^n: \G_m^n=(\Phi_r^n)^{-1}( \G_m^n) \to \G_m^n$ the restriction of $\Phi_r^n$ to the torus.   By analogy with the notation $X_r=(\Phi_r^n)^{-1}(X)$, we write $\phi_r^n: (\G_m^n)_r \to \G_m^n$. We also use the same notation $U_r=(\Phi_r^n)^{-1}(U)$  for the open $U$.  
The projection $q: \T\to \G_m$ to any of the factors has the property that  $q\circ \phi_r^n$ factors through  $\phi_r^1$.  For $U \subset \T\cap X$ open dense,
 there is  a projection $q: U \rightarrow \G_m$  to one of the factors which
is dominant.  As $U$ is irreducible, all the fibers of $q$ have dimension $\le (d-1)$.

The composite map $U_r\xrightarrow{\phi_r^n} U\xrightarrow{q} \G_m$  factors through $\phi_r^1: (\G_m)_r\to \G_m$, defining $q_r: U_r\to (\G_m)_r$. Concretely, if $q(\frac{x_1}{x_0},\ldots, \frac{x_n}{x_0})=\frac{x_1}{x_0}$ (say), then 
$q_r(\frac{x_1}{x_0},\ldots, \frac{x_n}{x_0})=\frac{x_1}{x_0}$.

Thus  one has a
commutative diagram
\ga{1}{ \xymatrix{ U_r \ar[r] \ar[dr]_{q_r}   \ar@/^1pc/[rr]^{\phi_r^n}  & \ar[d]_{q'_r} U'_r \ar[r] & \ar[d]^q U\\
& (\G_m)_r \ar[r]_{\phi^1_r} & \G_m
}\notag}
where $U'_r= U\times_{\G_m, \phi^1_r} (\G_m)_r,  \  q'_r=(\phi_r^1)^*q.$
\subsection{Constructibility} \label{ss:cons}

Recall that  $R^jq_! \F_p$ is constructible, see  \cite[Thm.~5.3.5]{Del73}. As $\phi_r^1$ is proper,  for any $j\in\N$, 
$\phi_r^{1*}(R^jq_! \F_p) $ is constructible as well and one has 
 a morphism
\ga{}{  \phi_r^{1*}(R^jq_! \F_p)\to R^jq_{r !} \F_p \notag}
of constructible sheaves, which for any $i\in \N ,$ induces an $\F_p$-linear map 
\ml{}{ \phi_r^{n*}: H^i_c(\G_m, R^jq_! \F_p) \to H^i_c((\G_m)_r, \phi_r^{1*}(R^jq_! \F_p)) 
\to \\ H^i_c((\G_m)_r, R^jq_{r!} \F_p) =  H^i_c( \G_m, \phi_{r*}^1 R^jq_{r!} \F_p) . \notag}
Here the left map is defined by adjunction.  \\[.1cm]

  We pose the  \\
    {\bf Induction hypothesis on $d'$:} Given  a   subscheme $X\subset \PP^n$ of dimension $d'$ and  a Zariski open subscheme $U\subset \T\cap X$ which is dense in $X$, there is a natural number $r_0$,  such that for all $r\ge r_0$, the map $\phi_r^{n*}: H^i_c(U)\to H^i_c(U_r)$ vanishes for all $i> d'$. \\[.5cm]

    The induction hypothesis is  trivially verified for $d'=0$. In the sequel we assume it is verified for $d'\le d-1$.

\begin{lemma} \label{lem:locsys}
 With the assumption~\ref{ass} on $X$ and ($U,q)$ as in \ref{ss:geom}, 
for $j>d-1$, there is an $ r_1\in \N $ such that  for all $r\ge r_1$ 
\ga{}{ \phi_r^{n*}: R^jq_! \F_p \to  \phi_{r*}^1R^jq_{r !} \F_p \notag}
vanishes. 

\end{lemma}
\begin{proof}
By \cite[Thm.~5.2.8]{Del73}, 
$R^jq_! \F_p$  verifies
  base change with stalks  $(R^jq_! \F_p)_x=H^j_c(q^{-1}(x))$ on  geometric points
  $x\in \G_m$. Thus we can apply the induction hypothesis on the dimension of the fibers $q^{-1}(x)$. 
  As $(\phi_r^n)^{-1} (q^{-1}(x))= q_r^{-1} ((\phi_r^{1 })^{-1}(x))$, 
 it follows that
the map
  \ga{}{ \phi_r^{n*}: H^j_c(q^{-1}(x))\to    H^j_c(q_r^{-1} (\phi_r^{1 })^{-1}(x)) \notag}
  vanishes by induction for $r \ge r(x)$ large enough depending on $x$. 
  Taking $x$ to be the geometric  generic point ${\rm Spec}(\overline{k(\G_m)})$ defines $r=r(x)$. 
  If $\sU\subset (\G_m)_r$  is a dense open on which $R^jq_{r !} \F_p$ is a local system, which is  lying in the smooth locus $\sU^0$ of $\phi_r^1$, then 
$\cap_g g^*\sU$,  for $g$ in the Galois group  $\Z/p^r$ of $\sU^0/\phi_r^1(\sU^0)$,  is Galois invariant and dense in $(\G_m)_r$, thus of the shape $(\G_m^0)_r$ for some dense open $\G_m^0\subset \G_m$. Then for all closed points $x\in \G_m^0$, we may take $r$ constant equal to $r(x)$. 
  We take $r_1$ greater or equal to $r$  and to the finitely many $r(x)$ for $x$ closed in $\G_m \setminus  \G_m^0$.    
  This finishes the proof. 
\end{proof}

\subsection{Representation theory}
With the assumption~\ref{ass} on $X$ and ($U,q)$ as in \ref{ss:geom}, 
we fix some  $j$ and consider a dense open $\G_m^0\subset \G_m$ over which $R^jq_! \F_p$ is a local system.
As $ \G_m\setminus \G_m^0$ is $0$-dimensional,  the excision map 
$H^2_c(\G_m^0,  R^jq_! \F_p)\to  H^2_c(\G_m, R^jq_! \F_p)$ is an isomorphism.

\begin{proposition} \label{prop:H2}
 There is an $r_2\in \N$ such that for all $j\in \N$,  all $r\ge r_2$,
 \ga{}{ \phi_r^{1*}: H^2_c(\G_m, R^jq_! \F_p) \to
 H^2_c((\G_m)_r, R^jq'_{r!} \F_p) \notag}
vanishes.

\end{proposition}
\begin{proof}
On $\G_m^0$, we denote by $\sV$ the  local system of $\F_p$-vector spaces  dual to $R^jq_! \F_p$. By classical duality, the cup-product 
$H^2_c(\G^0_m, \sV^\vee)\times H^0(\G^0_m, \sV)\to H^2_c(\G_m, \F_p)$ is a perfect duality. 
On the other hand, on $(\G_m^0)_r$, base change again implies
$$\phi^{1*}_r R^jq_! \F_p =
R^jq'_{r!}  \F_p, $$
thus $\phi^{1*}_r\sV$ is the local system dual to $R^jq'_{r!}  \F_p$.
Thus
 \ga{}{ \phi_r^{1*}: H^2_c(\G_m, R^jq_! \F_p) \to
 H^2_c((\G_m)_r, R^jq'_{r!} \F_p) \notag}
 is dual to the trace map
  \ga{}{ {\rm Tr}(\phi_r^1): H^0((\G^0_m)_r,  \phi^{1*}_r \sV) \to
 H^0(\G^0_m, \sV) \notag}
from which we show now that it vanishes for $r$ large.  As $ \sV$ is a local system, 
the dimension of $H^0((\G^0_m)_r,\phi^{1*}_r \sV)
$  as an $\F_p$-vector space is bounded above
by the rank of $R^jq_{*} \F_p$, and thus   does not depend on $r$.
  For $N$ a natural number, in $GL(N, \F_p)$ the order of a $p$-power torsion element is bounded by a constant depending on $N$ and $p$. 
Thus for $r$ large, the representation $\rho$ of  the Galois group $\Z/p^r$ of $\phi_r^1$  on
$H^0((\G^0_m)_r, \phi^{1*}_r \sV)$ 
cannot be faithful. Thus  $\rho$ factors as $\bar \rho$  through $\Z/p^s$ for some $s<r$.  This implies that 
for any $v\in H^0((\G^0_m)_r, \phi^{1*}_r \sV)$ 
\ga{}{{\rm Tr}(\phi_r^1)(v)=\sum_{i=0}^{p^r-1}\rho(i)(v)= \sum_{\bar i\in \Z/p^s}  \sum_{j=0}^{p^{r-s}-1} \rho( i+jp^s)(v)= p^{r-s}\big( \sum_{\bar i\in \Z/p^s}  \bar \rho(\bar i)(v)\big)=0. \notag
}
In the formula,  $i \in \Z/p^r$ and maps to $\bar i\in \Z/p^s$. This finishes the proof.
\end{proof}

\begin{corollary} \label{cor:H2}
 With $r_2$ as in Proposition~\ref{prop:H2}, 
 for  all $r\ge r_2$,
 \ga{}{ \phi_r^{n*}: H^2_c(\G_m, R^jq_! \F_p) \to
 H^2_c((\G_m)_r, R^jq_{r!} \F_p) \notag}
vanishes. 

\end{corollary}
\begin{proof} From the factorization $q'_r$ of $q_r$ over $q$, one 
obtains a factorization 
\ml{}{ \phi_r^{n*}: H^2_c(\G_m, R^jq_! \F_p) \xrightarrow{ \phi_r^{1*}   } H^2_c(\G_m,  R^jq'_{r!} \F_p) =
H^2_c(\G_m,  \phi_r^{1*}R^jq_{!} \F_p)\\
\to
 H^2_c((\G_m)_r, R^jq_{r!} \F_p) \notag}
where the first map is the one considered in Proposition~\ref{prop:H2} and the second one comes by functoriality
$R^jq'_{r!} \F_p\to R^jq_{r!}\F_p$.  This finishes the proof. 
\end{proof}

\begin{remark} \label{rmk:ref1}
We could have taken $U$ to be $\T$  in the Induction Hypothesis. In the proof, we have to consider the inverse image by $q$ of some dense open in $\G_m$ which is then a dense open in $\T$.   For this reason we just kept a neutral letter $U$.

\end{remark}

\section{Proof of Theorem~\ref{thm:main}}
We  argue by induction on $d$, starting with $d=0$.  We use the notations of the previous sections,
  make the assumption~\ref{ass} on $X$ and take $(U,q)$ as in \ref{ss:geom}.

We  consider the Leray spectral sequence for $q$ and $H^i_c(U)$ for $i>d$. 
One first has the corner map 
\ga{}{ H^i_c(U)\to E_2^{0i}(q)= H^0_c(\G_m, R^iq_! \F_p). \notag}
As $i>d>d-1$ we apply Lemma~\ref{lem:locsys}. 
Thus there is an $r_1$ 
such that  for all $r\ge r_1$,  the image of $ H^i_c(U)$  in $ H^i_c(U_r)$ lies in a subquotient of 
\ga{}{ E_2^{1,i-1}(q_r) = H^1_c((\G_m)_{r}, R^{i-1}q_{r!} \F_p) .\notag}
As $i-1>d-1$ the same argument shows that there is an $r'_1 \ge r_1$ such that for all $r\ge r'_1$,
the image of  $ H^i_c(U)$ in $H^i_c(U_r)$ lies in the image  of
\ga{}{ E_2^{2,i-2}(q_r)= H^2_c((\G_m)_{r}, R^{i-2}q_{{r} !} \F_p) .\notag}
If $i>d+1$ then $i-2>d-1$ and one applies again the same argument which finishes the proof. Or else one applies the following argument. 
If $i \ge d+1$, then $i-2 \ge d-1$, Corollary~\ref{cor:H2} implies that there is an $r_2\ge r'_1$ such that for all $r\ge r_2$,   the image of $ H^i_c(U)$  in $H^i_c(U_r)$ is $0$. 

As in the whole argument it does not matter whether we start the proof for $U$ or for $U_{r_0}$ for some given  natural number $r_0$, we in fact proved $\varinjlim_r H^i_c(U_r)=0$ for $i>d$.  
This finishes the proof of Theorem~\ref{thm:main}.

\section{Remarks}
\begin{itemize}

\item[1)] If $Z\subset \PP^n$ is any locally closed subscheme, with compactification  $\bar Z  \subset \PP^n$, 
applying again the excision exact sequence 
$$ \ldots \to  H^i_c(Z)\to H^i(\bar Z)\to H^i(\bar Z\setminus Z)\to H^{i+1}_c(Z) \to \ldots $$ 
 one sees that Theorem~\ref{thm:scholze} implies (and in fact is equivalent to) 
$$\varinjlim_r H^i_c(Z_r)=0 \ {\rm  for \ all } \ i > d$$
where $Z_r=(\Phi_r^n)^{-1}(Z)$. 
\item[2)]   Theorem~\ref{thm:main} can be expressed by writing $H^i(X, \F_p)$ as $H^i(\PP^n, \sF)$ where $\sF$ is the constructible sheaf $i_* \F_{p, X}$, where $i: X\hookrightarrow \PP^n$ is the closed embedding,  $\F_{p,X}$ is the constant \'etale sheaf $\F_p$ on $X$, and writing $H^i(X_r, \F_p)$ as $H^i(\PP^n, (\Phi_r^n)^*\sF)$.    More generally, we can take in Theorem~\ref{thm:main} $\sF$ to be any constructible sheaf.
\begin{thm}
If $k$ is an algebraically closed field of characteristic not equal to $p$,  and $\sF$ is a constructible sheaf of $\F_p$-vector spaces on $\PP^n$ with support of dimension $d$, then for $ i>d$, one has 
$$\varinjlim_r H^{i}((\PP^n)_r,  (\Phi_r^n)^*\sF) = 0 .$$

\end{thm}
The proof is exactly the same and we do not write the details. 
\item[3)] 
It may happen that even if $X$ is smooth, one needs $r\ge 2$ in Theorem~\ref{thm:main}. 
For example if  $p=2$, and $X$ is a smooth conic in $\PP^2$ such that the $H_i, \ i=0,1,2$ are tangent to $X$, then $X_1$ splits entirely into the union of four lines.
So the minimum $r$ which kills the whole cohomology $H^i(X), i>d$ is perhaps an intriguing geometric  invariant 
of the triple $$\big(X, \PP^n, (x_0:\ldots:x_n)\big).$$
\item[4)] Beilinson remarked that  Theorem~\ref{thm:main} has some relation to  \cite[Thm.~4.5.1]{BBDG82}. In \cite{BBDG82}, {\it loc. cit.} the formulation is with $\bar \Q_p$-coefficients but they immediately go down to $\F_p$ in the proof. They do not prove vanishing and they are not in $\PP^n$, but they bound the growth of the dimension of the  cohomology in  \'etale towers in function of the degree of the covers. 
\end{itemize}

\end{document}